\documentclass[a4paper,11pt]{article}
\usepackage[latin1]{inputenc}
\usepackage[english]{babel}
\usepackage{amsmath}
\usepackage{amsfonts}
\usepackage{amssymb}
\usepackage{epsfig}
\usepackage{amsopn}
\usepackage{amsthm}
\usepackage{color}
\usepackage{graphicx}
\usepackage{enumerate}
\newtheorem{theorem}{Theorem}[section]
\newtheorem{corollary}[theorem]{Corollary}
\newtheorem{lemma}[theorem]{Lemma}
\newtheorem{proposition}[theorem]{Proposition}
\newtheorem{definition}[theorem]{Definition}

\newtheorem*{theorem*}{Theorem}
\newtheorem*{lemma*}{Lemma}
\newtheorem*{remark*}{Remark}
\newtheorem*{definition*}{Definition}
\newtheorem*{proposition*}{Proposition}
\newtheorem*{corollary*}{Corollary}
\numberwithin{equation}{section}
\newcommand{\real}{\mathbb{R}}

\let\ced=\c         % cedilla
\def\a{\alpha}

\def\e{\varepsilon}        % Also, \varepsilon
\def\u{\upsilon}

\def\U{\Upsilon}

\def\cd{{\cal D}}

\def\qed{\,\unskip\kern 6pt \penalty 500
\raise -2pt\hbox{\vrule \vbox to8pt{\hrule width 6pt
\vfill\hrule}\vrule}\par}
\definecolor{darkblue}{rgb}{0.05, .05, .65}
\definecolor{darkgreen}{rgb}{0.1, .65, .1}
\definecolor{darkred}{rgb}{0.8,0,0}
\newcommand{\beqn}{\begin{equation}}
\newcommand{\eeqn}{\end{equation}}
\newcommand{\bear}{\begin{eqnarray}}
\newcommand{\eear}{\end{eqnarray}}
\newcommand{\bean}{\begin{eqnarray*}}
\newcommand{\eean}{\end{eqnarray*}}
\begin{document}
%%%%%%%%%%%%%%%%%%%%%%%%%%%%%%%%%%%%%%%%%%%%%%%%%

%%%%%%%%%%%%%%%%%%%%%%%%%%%%%%%%%%%%%%%%%%%%%%%%%
\title{\huge \bf Asymptotic behavior for the critical nonhomogeneous porous medium equation in low dimensions}

\author{
\Large Razvan Gabriel Iagar\,\footnote{Instituto de Ciencias
Matem\'aticas (ICMAT), Nicol\'as Cabrera 13-15, Campus de
Cantoblanco, 28049, Madrid, Spain, \textit{e-mail:}
razvan.iagar@icmat.es}, \footnote{Institute of Mathematics of the
Romanian Academy, P.O. Box 1-764, RO-014700, Bucharest, Romania.}
\\[4pt] \Large Ariel S\'{a}nchez,\footnote{Departamento de Matem\'{a}tica
Aplicada, Ciencia e Ingenieria de los Materiales y Tecnologia
Electr\'onica, Universidad Rey Juan Carlos, M\'{o}stoles,
28933, Madrid, Spain, \textit{e-mail:} ariel.sanchez@urjc.es}\\
[4pt] }
\date{}
\maketitle

\begin{abstract}
We deal with the large time behavior for a porous medium equation
posed in nonhomogeneous media with singular critical density
$$
|x|^{-2}\partial_tu(x,t)=\Delta u^m(x,t), \quad (x,t)\in
\real^N\times(0,\infty), \ m\geq1,
$$
posed in dimensions $N=1$ and $N=2$, which are also interesting in
applied models according to works by Kamin and Rosenau. We deal with
the Cauchy problem with bounded and continuous initial data $u_0$.
We show that in dimension $N=2$, the asymptotic profiles are
self-similar solutions that vary depending on whether $u_0(0)=0$ or
$u_0(0)=K\in(0,\infty)$. In dimension $N=1$, things are strikingly
different, and we find new asymptotic profiles of an unusual mixture
between self-similar and traveling wave forms. We thus complete the
study performed in previous recent works for the bigger dimensions
$N\geq3$.
\end{abstract}

\

\noindent {\bf AMS Subject Classification 2010:} 35B33, 35B40,
35K10, 35K67, 35Q79.

\smallskip

\noindent {\bf Keywords and phrases:} porous medium equation,
non-homogeneous media, singular density, asymptotic behavior,
radially symmetric solutions, nonlinear diffusion.

\section{Introduction}

The goal of this work is to study the asymptotic behavior of
solutions to the following nonhomogeneous porous medium equation
with critical singular density
\begin{equation}\label{eq1} |x|^{-2}\partial_{t}u(x,t)=\Delta
u^m(x,t), \quad (x,t)\in\real^{N}\times(0,\infty),
\end{equation}
where $m\geq1$ and we restrict ourselves to low dimensions $N=1$ and
$N=2$. This completes the panorama of the large time behavior of
solutions to Eq. \eqref{eq1}, already studied in dimensions $N\geq3$
in previous recent works by the authors \cite{IS13, IS14}. We will
work with the Cauchy problem for initial data satisfying
\begin{equation}\label{id}
u(x,0)=u_0(x)\in L^{\infty}(\real^N)\cap C(\real^N), \ u_0\geq0, \
u_0\not\equiv0, \quad {\rm for \ any} \ x\in\real^N.
\end{equation}
Moreover, unless we explicitly state the contrary, we will always
consider \emph{radially symmetric} initial conditions, that is,
$u_0(x)=u_0(r)$, $r=|x|$.

Equations such as
\begin{equation}\label{eq2}
\varrho(x)\partial_{t}u(x,t)=\Delta u^m(x,t), \quad
(x,t)\in\real^{N}\times(0,\infty),
\end{equation}
where $\varrho$ is a density function with suitable behavior, have
been proposed by Kamin and Rosenau in several classical papers
\cite{KR81, KR82, KR83} as a model for thermal propagation by
radiation in non-homogeneous plasma. Afterwards, a big development
of the mathematical theory associated to Eq. \eqref{eq2} begun,
assuming that the density satisfies
$$
\varrho(x)\sim|x|^{-\gamma}, \quad \hbox{as} \ |x|\to\infty,
$$
for some $\gamma>0$, as for example in the following papers
\cite{E90, EK94, GKKV, KT04, RV06, Ted, RV08, RV09, KRV10, NR} where
its qualitative properties and asymptotic behavior are studied. In
particular, some of these works were also considering purely
singular densities such as $\varrho(x)=|x|^{-\gamma}$, and proving
that asymptotic profiles for \eqref{eq2} come from explicit
solutions to equations with singular density. Moreover, it came out
that the density $\varrho(x)=|x|^{-2}$ (or $\varrho(x)\sim|x|^{-2}$
as $|x|\to\infty$) is critical for both the qualitative properties
(well-posedness, regularity) and large time behavior: indeed, for
$\gamma\in(0,2)$, properties do not depart from those already well
investigated of the porous medium equation $u_t=\Delta u^m$, while
for $\gamma>2$, they are different.

Restricting ourselves to the singular case
$\varrho(x)=|x|^{-\gamma}$, its theory developed later, due to the
difficulties involved by the presence of the singular coefficient.
Some results on qualitative behavior were established in
\cite{KT04}, then later in \cite[Section 6]{KRV10} and \cite[Section
3]{NR}, the latter two using weighted spaces for well-posedness, and
restricting the study to dimensions $N\geq3$. On the other hand,
formal transformations and explicit solutions were established in
\cite{GKKV, IRS}, where it becomes clear why $\gamma=2$ is a
critical exponent. Restricting to \eqref{eq1}, the large time
behavior in dimensions $N\geq3$ has been established recently by the
authors \cite{IS13, IS14}, and the results were quite striking: the
presence of both critical density $\gamma=2$ and the singularity at
$x=0$ led to many new and unexpected mathematical phenomena, that in
the general, non-critical and/or non-singular case do not happen. We
recently learned that also the nonhomogeneous equation for the
fractional porous medium
\begin{equation*}
\varrho(x)\partial_tu=(-\Delta)^s(u^m), \quad m>1, \ s\in(0,1),
\end{equation*}
has been proposed in \cite{Grillo15}, where the qualitative
properties and large time behavior are studied for suitable
$\varrho(x)$, but again excluding the critical density (that in the
fractional case holds for $\gamma=2s$).

\subsection{Main results}

We are now ready to state the main results of this paper. As already
explained, we deal with radially symmetric solutions $u(x,t)$ to Eq.
\eqref{eq1} having initial condition $u_0$ as in \eqref{id}. We
refer the reader to Section \ref{sec.prel} for the well-posedness of
the Cauchy problem \eqref{eq1}-\eqref{id} in our framework, and also
to some changes of function and variable converting \eqref{eq1} into
other nonlinear diffusion equations. As also seen in our precedent
works \cite{IS13, IS14}, an important difference in the asymptotic
behavior will appear with respect to the value of $u_0$ at the
origin: profiles are very different when $u_0(0)=0$ or $u_0(0)=K$.

But the {\bf main motivation} for us to write this paper is that the
large time behavior \emph{differs according to the dimension} $N=1$
and $N=2$ and in some cases \emph{strikingly departs} from the
parallel results in dimension $N\geq3$ given in \cite{IS14}. We
state in the sequel the results corresponding to $m>1$, as the case
$m=1$ is easy and is reduced to a comment in Section
\ref{sec.final}.

\medskip

\noindent \textbf{Large time behavior in dimension $N=2$.} The
asymptotic profiles to Eq. \eqref{eq1} with $m>1$ and in dimension
$N=2$ are of self-similar type, but with logarithmic variables. This
is explained by the fact that there exists a one-to-one
correspondence between radially symmetric solutions to \eqref{eq1}
and solutions to the standard porous medium equation
\begin{equation}\label{PME}
\partial_tw(s,t)=\partial_{ss}(w^m)(s,t), \qquad
(s,t)\in\real\times(0,\infty),
\end{equation}
in one dimension. We next state the precise results.
\begin{theorem}\label{th.conv1}
Let $u$ be a solution to \eqref{eq1} with $m>1$ and posed in
dimension $N=2$, with radially symmetric initial condition $u_0$,
satisfying \eqref{id} and moreover the following integral condition
(in radial variables)
\begin{equation}\label{int.cond}
u_0(0)=0, \qquad \int_0^{\infty}\frac{u_0(r)}{r}\,dr<\infty.
\end{equation}
Then we have
\begin{equation}\label{conv11}
\lim\limits_{t\to\infty}t^{\a}\|u(\cdot,t)-U_0(\cdot,t)\|_{\infty}=0,
\qquad \a=\frac{1}{m+1},
\end{equation}
where
\begin{equation}\label{profile11}
U_0(x,t)=t^{-\a}\left[C_0-k\left(\frac{\log\,|x|}{t^{\a}}\right)^2\right]_+^{1/(m-1)},
\qquad k=\frac{m-1}{2(m+1)},
\end{equation}
and $C_0=C(u_0)$ is uniquely determined by the initial condition
$u_0$.
\end{theorem}
In the case when $u_0(0)=K>0$, the asymptotic profile is
non-explicit but nevertheless it still has self-similar form. More
precisely
\begin{theorem}\label{th.conv2}
Let $u$ be a solution to \eqref{eq1} with $m>1$ and posed in
dimension $N=2$, with radially symmetric initial condition $u_0$,
satisfying \eqref{id} and moreover
$$u_0(0)=K>0, \qquad \lim\limits_{|x|\to\infty}u_0(x)=0.$$
Then there exists a radially symmetric profile
\begin{equation}\label{profile12}
W(x,t)=f\left(\frac{\log\,|x|}{\sqrt{t}}\right), \qquad W(0,t)=K, \
{\rm for \ any} \ t>0,
\end{equation}
such that
\begin{equation}\label{conv12}
\lim\limits_{t\to\infty}\|u(\cdot,t)-W(\cdot,t)\|_{\infty}=0.
\end{equation}
In particular, we also get that $u(0,t)=K$ for any $t>0$, that is,
the value at the origin do not change along the evolution.
\end{theorem}
\noindent \textbf{Remarks. 1.} Let us notice that, while Theorem
\ref{th.conv1} requires some control of the tail of $u_0$ (via
\eqref{int.cond}), in Theorem \ref{th.conv2} there is no need for
such control, only that $u_0(r)\to0$ as $r\to\infty$. This is
explained by the fact that, when $u_0(0)=K>0$, there is no time
decay rate $t^{-\a}$ and thus, the global large time behavior is
dominated by the \emph{inner behavior}, that is, in dynamic regions
close to $x=0$.

\noindent \textbf{2.} The profile $W$ (or $f$) in \eqref{profile12}
comes from a special self-similar solution to \eqref{PME} which is
not explicit, but it was introduced via a scaling process in
\cite{AR84, V84}. More details about it in Subsection
\ref{subsec.N2K} below, where its properties are recalled.

\noindent \textbf{3.} Results in Theorem \ref{th.conv2} remind of
those for the similar equation in dimension $N\geq3$, see
\cite[Theorem 1.6]{IS14}, although the profile there is explicit.
Meanwhile, Theorem \ref{th.conv1} departs strongly from the results
in dimensions $N\geq3$, \cite[Theorem 1.4]{IS14}, as there the
profiles were coming via an \emph{asymptotic simplification} to a
conservation law leading to \emph{peak-form} profiles, while in the
present case we have a regular, self-similar type function.

\medskip

\noindent \textbf{Large time behavior in dimension $N=1$.} In this
case, things are surprisingly different, although in appearance the
equation is the same. The is understood via some natural change of
variable indicated at the end of Section \ref{sec.prel}, leading to
a porous medium equation with convection:
\begin{equation}\label{PMEconv}
\partial_tw(s,t)=\partial_{ss}(w^m)(s,t)-\partial_{s}(w^m)(s,t),
\qquad (s,t)\in\real\times(0,\infty),
\end{equation}
whose behavior (of traveling waves type) is essentially different
from the one for the standard porous medium equation \eqref{PME}.
There is once more a difference on whether $u_0(0)=0$ or $u_0(0)>0$.

The case of initial conditions $u_0$ with $u_0(0)=0$ is largely
similar to the equivalent case in \cite{IS14}. As there, we will get
an asymptotic profile presenting a \emph{peak-type discontinuity},
thus, the uniform convergence will be replaced by a slightly weaker
concept, the \textbf{convergence in the sense of graphs} for
multivalued functions, which allows us to deal with jump
discontinuities. For the sake of completeness, we recall here this
concept (following \cite{EVZ93, IS14}). Let $f$,
$g:D\subseteq\real\mapsto2^{\real}$ be two multivalued functions. We
define the \emph{distance between their graphs} as
$$
d_g(f(x),g(x))=\inf\{|y-z|:y\in f(x), \ z\in g(x)\}, \qquad {\rm for
\ any } \ x\in D.
$$
Let $\{f_k\}:D\subseteq\real\mapsto2^{\real}$ be a sequence of
multivalued functions and $F:D\subseteq\real\mapsto2^{\real}$. We
say that ${f_k}$ \emph{converges in the sense of graphs} to $F$ if
for any $\e>0$, there exists $k(\e)>0$ sufficiently large such that
$$
d_g(f_k(x),F(x))\leq\e, \qquad {\rm for \ any} \ k\geq k(\e), \ x\in
D.
$$
Notice that, in the standard case of univalued functions, this
notion reduces to the usual uniform convergence. In the case of a
function $F$ having a jump discontinuity at $x_0\in D$ (as it will
be our case below), letting
$$
l_{-}:=\lim\limits_{x\to x_0^{-}}F(x)<\lim\limits_{x\to
x_0^{+}}F(x)=l_{+},
$$
we will work with $F(x_0)=[l_{-},l_{+}]$. We define, for a
measurable function $h$, the norm
$$
\|h\|_{p,1}:=\left(\int_{\real}\frac{|h(x)|^p}{|x|}\,dx\right)^{1/p}.
$$
With these definitions, we are now ready to state our convergence
result.
\begin{theorem}\label{th.conv3}
Let $u$ be a solution to \eqref{eq1} with $m>1$, posed in dimension
$N=1$, with radially symmetric initial condition $u_0$ satisfying
\eqref{id} and moreover
\begin{equation}\label{idminus}
u_0(0)=0, \qquad M_0:=\int_{0}^{\infty}\frac{u_0(x)}{x}\,dx<\infty.
\end{equation}
Then:

\noindent (a) For any $p\in[1,\infty)$, we have
\begin{equation}\label{pconv}
\lim\limits_{t\to\infty}t^{(p-1)/mp}\|u(\cdot,t)-F(\cdot,t)\|_{p,1}=0,
\end{equation}
where
\begin{equation}\label{profile21}
F(x,t)=\left\{\begin{array}{ll}t^{-1/m}\left[\frac{1}{m}t^{-1/m}\log\,|x|\right]_+^{1/(m-1)},
& {\rm for} \ 0<|x|<e^{kt^{1/m}},\\ 0, & {\rm for} \ |x|\geq
e^{kt^{1/m}}.\end{array}\right.
\end{equation}
and $k=k(M_0)$ depends only on $M_0$.

\noindent (b) Introducing self-similar changes of variable and
function
\begin{equation}\label{ssvar}
\overline{u}(y,t)=t^{1/m}u(x,t), \qquad
\overline{F}(y)=t^{1/m}F(y,t)=\left[\frac{y}{m}\right]^{1/(m-1)}, \
y=t^{-1/m}\log\,|x|,
\end{equation}
we have
\begin{equation}\label{conv21}
\overline{u}(y,t)\longrightarrow\overline{F}(y), \qquad {\rm as} \
t\to\infty,
\end{equation}
the convergence being in the sense of graphs.
\end{theorem}
\noindent \textbf{Remarks. 1.} It is obvious that we cannot get
uniform convergence in \eqref{pconv} (by letting $p\to\infty$), as
the limit profile $F$ presents a \emph{peak-type discontinuity} and
it cannot thus be a uniform limit of continuous functions. The
convergence in the sense of graphs replaces the uniform convergence
here.

\noindent \textbf{2.} The results are quite similar to
\cite[Theorems 1.1 and 1.4]{IS14}. The main difference with respect
to there is that the direction of evolution of the profile is
\emph{forward}, as in dimension $N\geq3$ it evolves backward.

\noindent \textbf{3. Faster decay in the \emph{inner region}.} The
fact that the limit profile vanishes for $x\in[0,1]$ only means that
in that region, the solutions decay in time faster than the global
time scale $t^{-1/m}$. Indeed, we infer from the self-map
(inversion) for \eqref{eq1} between dimensions $N=3$ and $N=1$ (see
\cite[Section 2.1]{IRS}), and from the large time behavior in outer
regions for $N\geq3$ \cite[Theorem 1.1]{NR} and \cite[Theorem
1.11]{IS14}, that the large time behavior in \emph{inner regions}
near $x=0$ states
$$
\lim\limits_{t\to\infty}t^{1/(m-1)}\|u(\cdot,t)-B_{D}(\cdot,t)\|_{\infty}=0,
\qquad {\rm uniformly \ in \ sets} \ \{|x|\leq\delta t^{-1/(m-1)}\},
$$
for any $\delta>0$, where $B_D$ is an explicit profile,
$$
B_D(x,t)=t^{-1/(m-1)}\left[D+\frac{1}{m}\log\left(|x|t^{1/(m-1)}\right)\right]_{+}^{1/(m-1)},
$$
for some $D>0$ depending only on the initial condition $u_0$. This
shows in particular that the time decay rate in \emph{inner regions}
of the form above is faster, as expected ($t^{-1/(m-1)}$ instead of
the global $t^{-1/m}$), whence the uniform vanishing near $x=0$ of
the global profile $F$ in \eqref{profile21}.

\noindent \textbf{4. } The appearance of a \emph{peak-type profile}
is a very surprising and interesting result at first sight. It is
explained by the fact that Eq. \eqref{PMEconv} lies in a regime
where the convection term dominates on the large-time evolution,
producing limit profiles characteristic to conservation laws, not to
diffusion equations.

When $u_0(0)=K>0$, the asymptotic profiles are of a "special
self-similar" type. This comes in fact from a traveling wave profile
in transformed variables. More precisely,
\begin{theorem}\label{th.conv4}
Let $u$ be a solution to \eqref{eq1} with $m>1$, posed in dimension
$N=1$, with radially symmetric initial condition $u_0$ satisfying
\eqref{id} and moreover
\begin{equation}\label{idplus1}
u_0(0)=K>0, \qquad \lim\limits_{|x|\to\infty}u_0(x)=0, \qquad 0\leq
u_0(x)\leq K, \ {\rm for \ any} \ x\in\real,
\end{equation}
and
\begin{equation}\label{idplus}
u_0\in C^{0,\alpha}(\real), \ {\rm for \ some} \ \alpha\in(0,1].
\end{equation}
Then, there exists $x_0>0$ (depending only on $u_0$) such that
\begin{equation}\label{conv22}
\lim\limits_{t\to\infty}\|u(\cdot,t)-U_{x_0}(\cdot,t)\|_{\infty}=0,
\end{equation}
where
\begin{equation}\label{profile22}
U_{x_0}(x,t)=\left[K^{m-1}-\left(\frac{|x|}{x_0}e^{-ct}\right)^{(m-1)/m}\right]_{+}^{1/(m-1)},
\qquad c=K^{m-1}.
\end{equation}
\end{theorem}
\noindent The dependence of $x_0$ on the initial condition $u_0$ is
given via an integral condition in transformed variables, see
\eqref{interm17} in Subsection \ref{subsec.N1K}.

\noindent \textbf{Remarks. 1.} Profiles \eqref{profile22} were first
obtained in \cite[Section 7]{GKKV}. The form of them is not standard
self-similar.

\noindent \textbf{2. } In fact, we have a family of profiles
$U_{x_0}$, for any $x_0\in\real$, and one can readily notice that
they are not equivalent, that is,
$$
\|U_{x_0}(\cdot,t)-U_{y_0}(\cdot,t)\|\not\to 0, \qquad {\rm as} \
t\to\infty,
$$
for $x_0\neq y_0$. Thus, it is a relevant fact that we can identify
the exact translation parameter $x_0$, only in terms of the initial
condition $u_0$.

\noindent \textbf{3. Result for non-radial solutions.} In dimension
$N=1$, we are able to extend the large time behavior to solutions
with general (not radially symmetric) data $u_0$. We postpone this
fact to the final section, see \eqref{conv.nonradial} for a result.

\medskip

\noindent \textbf{Organization of the paper.} We begin with a rather
standard section which gathers well-posedness and some preliminary
facts about the solutions to \eqref{eq1}, that are later used in the
proofs. We devote further Sections \ref{sec.N2} and \ref{sec.N1} to
the proofs of the main results Theorems \ref{th.conv1} and
\ref{th.conv2} (concerning dimension $N=2$), respectively Theorems
\ref{th.conv3} and \ref{th.conv4} (concerning dimension $N=1$). Both
sections are divided into symmetrical subsections, according to
whether $u_0(0)=0$ or $u_0(0)>0$. In general, the proofs for the
case $u_0(0)=0$ are quite short, while the cases when $u_0(0)>0$ are
the most interesting. We end the present work by Section
\ref{sec.final} of extensions and comments, in which, noticeably, we
complete the panorama of the large time behavior for data $u_0$
which are \emph{not radially symmetric in dimension $N=1$} and with
the (easy) linear case $m=1$.

\section{Preliminaries: well-posedness and change of
variable}\label{sec.prel}

The existence and uniqueness of solutions to \eqref{eq1} posed in
radially symmetric variables is granted by the results in
\cite[Section 5]{KT04}. We sketch them here for the sake of
completeness. Let us consider the more general Cauchy problem (with
radially symmetric general density function $\varrho$)
\eqref{eq2}-\eqref{id} with $u_0$ radially symmetric and $m\geq1$.
We state the following
\begin{definition}\label{def.sol}
By a radial solution to problem \eqref{eq2}-\eqref{id} in
$\real^N\times(0,\infty)$, we understand any nonnegative, radially
symmetric function $u$, continuous in $\real^N\times(0,\infty)$ such
that, for any rectangle $(r_1,r_2)\times(0,T)$ with $0<r_1<r_2$,
$T\in(0,\infty)$ we have
\begin{equation}\label{sol}
\begin{split}
\int_0^T\int_{r_1}^{r_2}&\left[\varrho
u\varphi_t+\left(\varphi_{rr}+\frac{N-1}{r}\varphi_r\right)u^m\right]r^{N-1}\,dr\,dt\\
&=\int_{r_1}^{r_2}\varrho\left[u(r,T)\varphi(r,T)-u_0(r)\varphi(r,0)\right]r^{N-1}\,dr\\
&+\int_0^T\left[r_2^{N-1}u^m(r_2,t)\varphi_r(r_2,t)-r_1^{N-1}u^m(r_1,t)\varphi_r(r_1,t)\right]\,dt,
\end{split}
\end{equation}
for any test function $\varphi\in C^{2,1}((r_1,r_2)\times(0,T))$,
$\varphi\geq0$, such that $\varphi(r_1,t)=\varphi(r_2,t)=0$ for any
$t\in[0,T)$.
\end{definition}
Notice that Definition \ref{def.sol} extends in fact the standard
notion of \emph{very weak solution} to a parabolic PDE (when all
derivatives are translated to the test function), employed for
example in \cite[Chapter 6.2]{VazquezPME}. In this general
framework, the well-posedness result is given by \cite[Theorem 5.2,
part (i)]{KT04}, which states the following:
\begin{proposition}\label{prop.wp}
Let $\varrho(r)\sim r^{-\gamma}$ both as $r\to0$ and as
$r\to\infty$. Then, in the previous conditions and notation, the
Cauchy problem \eqref{eq2}-\eqref{id} is well-posed in
$L^{\infty}([0,T];\real^N)\cap C({\real^N\times(0,\infty)})$ for any
$T<\infty$, if and only if either $N=2$ or $\gamma=2$. Letting in
particular $\varrho(r)=r^{-2}$, it follows that the Cauchy problem
\eqref{eq1}-\eqref{id} is well-posed.
\end{proposition}
The following standard result is useful in the sequel.
\begin{lemma}\label{lem.prel}
Let $u$ be a solution to Eq. \eqref{eq1}, with initial condition
$u_0$ as in \eqref{id} and moreover \emph{radially symmetric and
radially non-increasing}. Then, the solution $u(x,t)$ is also
radially symmetric and radially non-increasing at any time $t>0$
\end{lemma}
\begin{proof}
The radial symmetry follows from the invariance of Eq. \eqref{eq1}
to rotations. For the radial monotonicity, we use a standard
argument with a twist, that we only give at formal level. We write
Eq. \eqref{eq1} in radial variables as
$$
\partial_tu=r^2\partial_{rr}(u^m)+(N-1)r\partial_{r}(u^m),
$$
and, by letting $v(r,t)=u^m(r,t)$, we obtain
\begin{equation}\label{interm7}
\partial_tv=m\left[r^2v^{(m-1)/m}\partial_{rr}v+(N-1)rv^{(m-1)/m}\partial_rv\right],
\end{equation}
with the property that $\partial_rv(r,0)\leq0$, for any $r\geq0$. We
now consider $z:=\partial_rv$, differentiate Eq. \eqref{interm7}
with respect to $r$ and, after straightforward calculations (that we
leave to the reader), we get the equation satisfied by $z$:
\begin{equation}\label{interm8}
\begin{split}
\partial_tz&-mr^2v^{(m-1)/m}\partial_{rr}z-\left[2mrv^{(m-1)/m}+m(N-1)rv^{(m-1)/m}-(m-1)r^2v^{-1/m}z\right]z_r\\
&-\left[(N-1)mv^{(m-1)/m}z+(m-1)(N-1)rv^{-1/m}z^2\right]=0.
\end{split}
\end{equation}
From the assumption of monotonicity, we know that
$z(r,0)=\partial_rv(r,0)\leq0$, for any $r\geq0$. Since $z\equiv0$
is a solution to \eqref{interm8} and $z(0,t)=\partial_r(u^m)(0,t)=0$
for any $t>0$, by comparison we obtain that
$z(r,t)=\partial_r(u^m)(r,t)\leq0$, for any
$(r,t)\in[0,\infty)\times(0,\infty)$, which implies that $u^m(r,t)$,
hence also $u(r,t)$, is radially non-increasing for any $t>0$. All
the above is justified at formal level. A rigorous proof follows by
standard approximation following the method in \cite[Chapter
9.3]{VazquezPME}.
\end{proof}
In order to be able to pass from estimates in $L^1$ to estimates in
$L^{\infty}$ for the solutions in dimension $N=1$, we need some
further regularity of the solutions. This is insured by the
following
\begin{lemma}\label{lem.Holder}
Let $u$ be a non-negative, radially symmetric solution to
\eqref{eq1} in dimension $N=1$, such that its initial condition
satisfies $u_0\in C^{0,\alpha}(\real)$, for some $\alpha\in(0,1)$,
and moreover $u(0,t)=K>0$ for any $t\geq0$. Then $u(\cdot,t)\in
C^{0,\alpha}(\real)$ for any $t>0$, and the Holder constant is
uniformly bounded (that is, independent of time variable).
\end{lemma}
The unusual condition $u(0,t)=K$ for any $t>0$ is specific to Eq.
\eqref{eq1}, as shown in Subsections \ref{subsec.N2K} and
\ref{subsec.N1K} (see also \cite[Theorem 1.6]{IS14} for dimension
$N\geq3$).
\begin{proof}
The proof relies on the comparison principle. As the result is
obvious at the regular points of $u$, we will only prove it around
$x=0$. From hypothesis, we know that there exists $H>0$ such that
$$
|u_0(x)-K|\leq H|x|^{\alpha}, \qquad {\rm for \ any} \
x\in[0,\infty),
$$
or equivalently
$$
w_{-}(x):=\left[K-Hx^{\alpha}\right]_{+}\leq u_0(x)\leq
w_{+}(x):=K+Hx^{\alpha}, \qquad {\rm for \ any} \ x\in[0,\infty),
$$
the radial symmetry insuring that it is enough to work on the right
half-line. We next show that $w_{-}$ and $w_{+}$ are respectively a
sub- and a supersolution to \eqref{eq1} in
$(0,\infty)\times(0,\infty)$. Indeed, defining the operator
$$
\mathcal{L}u(x,t):=|x|^{-2}\partial_tu(x,t)-\Delta u^m(x,t),
$$
we have for $x>0$
\begin{equation*}
%\begin{split}
\mathcal{L}w_{-}(x,t)=-m(m-1)\left[K-Hx^{\alpha}\right]_{+}^{m-2}\alpha^2x^{2\alpha-2}+m\left[K-Hx^{\alpha}\right]_{+}^{m-1}\alpha(\alpha-1)x^{\alpha-2}\leq0,
%\end{split}
\end{equation*}
which shows that $w_{-}$ is a subsolution. On the other hand,
\begin{equation*}
\begin{split}
\mathcal{L}w_{+}(x,t)&=-m(m-1)\left[K+Hx^{\alpha}\right]^{m-2}\alpha^2x^{2\alpha-2}\\
&-m\left[K+Hx^{\alpha}\right]^{m-1}\alpha(\alpha-1)x^{\alpha-2}\\
&=-m\left[K+Hx^{\alpha}\right]^{m-2}\alpha
x^{\alpha-2}\left[r^{\alpha}((m-1)\alpha-H(1-\alpha))-K(1-\alpha)\right].
\end{split}
\end{equation*}
Since $\alpha<1$, one can choose $H_0>0$ sufficiently large such
that
$$
\alpha\leq\frac{H}{m-1+H}, \qquad {\rm for \ any } \ H\geq H_0,
$$
or equivalently, $(m-1)\alpha-H(1-\alpha)\leq0$. As $K(1-\alpha)>0$,
it readily follows that $\mathcal{L}w_{+}\geq0$. Since $w_{-}$ and
$w_{+}$ are stationary, and the comparison on the lateral boundary
$\{0\}\times(0,\infty)$ is insured by the fact that $u(0,t)=K$ for
any $t>0$, the comparison principle entails
$$
w_{-}(x)\leq u(x,t)\leq w_{+}(x), \qquad {\rm for \ any} \
(x,t)\in(0,\infty)\times(0,\infty),
$$
whence $u(\cdot,t)\in C^{0,\alpha}(\real)$ for any $t>0$. Moreover,
the Holder constant is uniformly bounded by the sufficiently large
$H_0$ chosen for $w_{+}$ to be a supersolution, and this constant
does not depend on time.
\end{proof}
\noindent \textbf{Remark.} The proof above does not directly allow
for $\alpha=1$, that is, Lipschitz data. But if $u_0$ is Lipschitz
and uniformly bounded, it is also Holder continuous for any
$\alpha<1$, since if $x$, $y\in\real$ such that $|x-y|\geq1$, then
$$
|u_0(x)-u_0(y)|\leq2\|u_0\|_{\infty}\leq2\|u_0\|_{\infty}|x-y|^{\alpha},
$$
for any $\alpha<1$, and if $|x-y|<1$, then
$$
|u_0(x)-u_0(y)|\leq L|x-y|\leq L|x-y|^{\alpha},
$$
for any $\alpha<1$. Thus, we get the following
\begin{corollary}\label{cor.Lip}
If, in the previous notation and hypothesis, $u_0$ is a Lipschitz
function, then $u(\cdot,t)\in C^{0,\alpha}(\real)$ for any
$\alpha\in(0,1)$ and $t>0$.
\end{corollary}

As it was already established in our previous works \cite{IS13,
IS14}, the study of Eq. \eqref{eq1} is strongly based on some
transformations at the level of radially symmetric variables, which
lead to equations where the effects of the singular coefficient are
transformed into absorption or convection effects, which are better
understood. We recall the transformations adapted to our special
cases.

\medskip

\noindent \textbf{1. The quasilinear case $m>1$}. We write
\eqref{eq1} in radial variables as
$$
r^{-2}\partial_tu(r,t)=\partial_{rr}(u^m)(r,t)+\frac{N-1}{r}\partial_{r}(u^m)(r,t)
$$
and make the change of variable and function $u(r,t)=w(s,t)$,
$r=e^s$. Then, the new function $w$ solves either the standard
porous medium equation in one dimension, if the starting dimension
for Eq. \eqref{eq1} was $N=2$
\begin{equation*}
\partial_tw(s,t)=\partial_{ss}(w^m)(s,t), \qquad
(s,t)\in\real\times(0,\infty),
\end{equation*}
or a porous medium equation with convection also in one dimension,
\begin{equation*}
\partial_tw(s,t)=\partial_{ss}(w^m)(s,t)-\partial_{s}(w^m)(s,t),
\qquad (s,t)\in\real\times(0,\infty),
\end{equation*}
if the starting dimension for Eq. \eqref{eq1} was $N=1$. More
details about the transformation (in the general case) are given in
\cite[Section 3]{IS14}.

\medskip

\noindent \textbf{2. The linear case $m=1$}. In this case, the
transformation given in \cite[Section 3]{IS13} works similarly; more
precisely, letting
\begin{equation}\label{transf.lin}
u(r,t)=w(s,t), \qquad s=\log\,r+(N-2)t,
\end{equation}
it follows that $w$ is a solution to the heat equation in one
dimension:
\begin{equation}\label{HE}
\partial_tw(s,t)=\partial_{ss}w(s,t), \qquad
(s,t)\in\real\times(0,\infty).
\end{equation}
Through these transformations, we transform our problem into
equations for which most of their features are by now well
understood. Proofs of the main results thus come from arguments at
the level of the transformed equations \eqref{PME}, \eqref{PMEconv}
or \eqref{HE}. In the sequel, we will work with $m>1$ and reduce the
easy linear case $m=1$ to a final comment in Section
\ref{sec.final}.

\section{Large time behavior in dimension $N=2$}\label{sec.N2}

We establish the asymptotic behavior of solutions to the Cauchy
problem \eqref{eq1}-\eqref{id} for $N=2$ and $m>1$, thus proving
Theorems \ref{th.conv1} and \ref{th.conv2}. As usual, we consider
$u_0$ radially symmetric, so that the solution $u$ is radially
symmetric, well defined, bounded and continuous (according to
Definition \ref{def.sol} and Proposition \ref{prop.wp}). As
specified in the Introduction and precedent work \cite{IS14}, the
analysis is divided on whether $u_0(0)=0$ or $u_0(0)=K>0$.

\subsection{Large time behavior for data with
$u_0(0)=0$}\label{subsec.N20}

As explained in Section \ref{sec.prel}, by doing the transformation
$w(s,t)=u(r,t)$, where $r=e^{s}$, we obtain the porous medium
equation \eqref{PME} in $\real\times(0,\infty)$. Moreover, the
initial condition becomes $w_0(s)=u_0(r)$, which, taking into
account \eqref{int.cond}, satisfies
$$
\int_{\real}w_0(s)\,ds=\int_0^{\infty}\frac{w_0(\log\,r)}{r}\,dr=\int_0^{\infty}\frac{u_0(r)}{r}\,dr<\infty,
$$
whence $w_0\in L^1(\real)$. By standard results in the theory of the
porous medium equation (see for example \cite{V}), we know that
\begin{equation}\label{conv1}
\lim\limits_{t\to\infty}t^{\a}|w(s,t)-B_0(s,t)|=0,
\end{equation}
uniformly in $\real$, where $B_0(s,t)$ is the Barenblatt
(fundamental) solution to \eqref{PME} with the same total mass as
$w_0$, more precisely
$$
B_0(s,t)=t^{-\a}\left[C_0-k\left(\frac{s}{t^{\a}}\right)^2\right]_+^{1/(m-1)},
\qquad \alpha=\frac{1}{m+1}, \ k=\frac{m-1}{2(m+1)},
$$
and the parameter $C_0>0$ is uniquely determined by the fact that
$$
\int_{-\infty}^{\infty}B_{0}(s,t)\,ds=\int_{-\infty}^{\infty}w_0(s)\,ds.
$$
Undoing the previous transformation $s=\log\,r$ and passing to
initial variables, we get
$$
\lim\limits_{t\to\infty}t^{\a}|u(r,t)-U_0(r,t)|=0, \qquad
U_0(r,t)=t^{-\a}\left[C_0-k\left(\frac{\log\,r}{t^{\a}}\right)^2\right]_+^{1/(m-1)},
$$
uniformly for $r\in[0,\infty)$, proving thus Theorem \ref{th.conv1}.

\subsection{Large time behavior for data with
$u_0(0)=K>0$}\label{subsec.N2K}

In this case, things are a bit more complicated. By the same
transformation $w(s,t)=u(r,t)$, $r=e^s$, we arrive to Eq.
\eqref{PME} but with initial data $w_0$ such that
\begin{equation}\label{interm1}
\lim\limits_{s\to-\infty}w_0(s)=K>0, \qquad
\lim\limits_{s\to\infty}w_0(s)=0.
\end{equation}
We need thus to find the asymptotic profile of the porous medium
equation with such initial data, which is obviously non-integrable
in $\real$. To this end, we have a non-explicit candidate, in
self-similar form, given in \cite{AR84} and \cite[Section 4]{V84} as
a particular case (with $\a=0$ in their notation) of a more general
family of solutions. More precisely, it is shown that there exists a
self-similar solution of the form
\begin{equation}\label{specialSS}
W(s,t)=f(st^{-\beta}), \qquad \beta=\frac{1}{2},
\end{equation}
having as initial trace as $t\to0$, the Heaviside function
$$
W_0(x):=W(x,0)=\left\{\begin{array}{ll}1, \quad {\rm for} \ x\leq0, \\
0, \quad {\rm for} \ x>0,\end{array}\right.
$$
We next show that $W$ is the asymptotic profile as $s\to\infty$ for
solutions to \eqref{PME} with initial data as in \eqref{interm1}
with $K=1$, the result being then extended to general $K>0$ by a
rescaling.

\begin{proof}[Proof of Theorem \ref{th.conv2}] The desired large-time
behavior for Eq. \eqref{PME} with data as in \eqref{interm1} is
proved in \cite{vDPel}, with a technique involving somehow complex
gradient and integral estimates. We present (in a rather sketchy
form) an alternative proof which relies on the established
\emph{four step method} (see for example \cite{V}), along the lines
of the parallel result in \cite[Section 5]{IS14}. We divide it into
several parts and we pass quickly through the standard ones.

\medskip

\noindent \textbf{Step 1. Rescaling.} For $w$ solution to
\eqref{PME} with initial data as in \eqref{interm1}, define the
family of rescaled functions
\begin{equation}\label{resc}
w_{\lambda}(s,t):=w(\lambda^{1/2}s,\lambda t).
\end{equation}
It is immediate to check that $w_{\lambda}$ is again a solution to
\eqref{PME}.

\medskip

\noindent \textbf{Steps 2-3. Uniform estimates and passage to the
limit.} We want to get suitable uniform estimates in order to be
able to pass to the limit as $\lambda\to\infty$. This follows easily
from the well-established theory of Eq. \eqref{PME}. On the one
hand, by immediate comparison, we get
$$
|w_{\lambda}(s,t)|=|w(\lambda^{1/2}s,\lambda
t)|\leq\|w_0\|_{\infty},
$$
since constants are solutions to \eqref{PME}. On the other hand, an
uniform modulus of continuity for \emph{bounded} solutions to
\eqref{PME} is well-known, see \cite{DBF} and \cite[Chapter
7]{VazquezPME}. By Arzela-Ascoli theorem, there exists (along a
subsequence) a limit $W_{\infty}(s,t)$ of $w_{\lambda}$, as
$\lambda\to\infty$, with uniform convergence on compact sets in
$\real$. It is a standard step to prove (see for example
\cite[Chapter 18]{VazquezPME}) that $W_{\infty}$ is a weak solution
to Eq. \eqref{PME}, satisfying the same uniform bounds as the family
$w_{\lambda}$.

\medskip

\noindent \textbf{Step 4. Identification of the limit.} It remains
to show that $W_{\infty}=W$, where $W$ is given in
\eqref{specialSS}. To this end, following ideas in \cite[Section
5]{IS14}, we want to show that $W_{\infty}$ takes as initial trace
when $t\to0$ a Heaviside-type function, that means
$$
\lim\limits_{t\to0}W_{\infty}(s,t)=\left\{\begin{array}{ll}1, \qquad
{\rm for} \ s\leq0, \\ 0, \qquad {\rm for} s>0, \end{array}\right.
$$
in the sense of distributions, or equivalently,
\begin{equation}\label{interm2}
\lim\limits_{t\to0}\int_{-\infty}^{\infty}\left[W_{\infty}(s,t)\Phi(s)\,ds-\int_{-\infty}^0\Phi(s)\,ds\right]=0,
\end{equation}
for any test function $\Phi\in\cd(\real)$ (recall that we are
working for the moment with the assumption $K=1$). Let
$\Phi\in\cd(\real)$ and estimate
\begin{equation*}
\begin{split}
\left|\int_{-\infty}^{\infty}(w_{\lambda}(s,t)-w_{\lambda}(s,0))\Phi(s)\,ds\right|&=\left|\int_{-\infty}^{\infty}\int_0^t\partial_{t}w_{\lambda}(s,\theta)\Phi(s)\,d\theta\,ds\right|\\
&=\left|\int_{-\infty}^{\infty}\int_0^t\partial_{ss}(w_{\lambda}^m)(s,\theta)\Phi(s)\,d\theta\,ds\right|\\
&=\left|\int_0^t\left[\int_{-\infty}^{\infty}w_{\lambda}^m(s,\theta)\Phi''(s)\,ds\right]\,d\theta\right|\\
&\leq C\|w_0\|_{\infty}|{\rm supp}\,\Phi|\|\Phi''\|_{\infty}t\to0,
\qquad {\rm as} \ t\to0,
\end{split}
\end{equation*}
where by $|{\rm supp}\,\Phi|$ we denote the Lebesgue measure of the
(compact) support of the test function $\Phi$. We have thus shown by
now that
\begin{equation}\label{interm3}
\lim\limits_{t\to0}\int_{-\infty}^{\infty}w_{\lambda}(s,t)\Phi(s)\,ds=\int_{-\infty}^{\infty}w_{\lambda}(s,0)\Phi(s)\,ds,
\end{equation}
for any $\Phi\in\cd(\real)$ and $\lambda>0$, the convergence being
uniform with respect to $\lambda$ in any interval of the form
$[\lambda_0,\infty)$ with $\lambda_0>0$. It still remains to show
that
$$
\lim\limits_{\lambda\to\infty}\int_{-\infty}^{\infty}w_{\lambda}(s,0)\Phi(s)\,ds=\int_{-\infty}^0\Phi(s)\,ds,
$$
for any $\Phi\in\cd(\real)$. By splitting the integral, we obtain
\begin{equation*}
\begin{split}
\int_{-\infty}^{\infty}w_{\lambda}(s,0)\Phi(s)\,ds&=\left(\int_{-\infty}^{0}+\int_{0}^{\infty}\right)w_0(\lambda
s)\Phi(s)\,ds\\&=\int_{-\infty}^{0}\Phi(s)\,ds+\int_{-\infty}^{0}(w_0(\lambda
s)-1)\Phi(s)\,ds+\int_{0}^{\infty}w_0(\lambda
s)\Phi(s)\,ds\\&=T_1+T_2+T_3,
\end{split}
\end{equation*}
and the convergence of $T_2$ and $T_3$ to 0 as $\lambda\to\infty$
(even on subsequences) follows from the Lebesgue's Dominated
Convergence Theorem and the boundedness of $w_0$. We omit the
details, which are very similar to the ones in \cite[Section
5]{IS14}. By uniqueness, we obtain that $W_{\infty}\equiv W$, the
self-similar solution in \eqref{specialSS}, and that $w_{\lambda}\to
W$ as $\lambda\to\infty$ (that means, not only on a subsequence). We
have thus proved that
$$
\lim\limits_{\lambda\to\infty}|w_{\lambda}(s,t)-W(s,t)|=0,
$$
uniformly for $(s,t)\in K\times(0,\infty)$ with $K$ compact subset
of $\real$, or equivalently
$$
\lim\limits_{\lambda\to\infty}|w(\lambda^{1/2}s,\lambda
t)-W(s,t)|=0.
$$
Letting first $t=1$ fixed, then relabeling $\lambda=t$ and $s\mapsto
t^{1/2}s$, we obtain that
$$
\lim\limits_{t\to\infty}|w(s,t)-W(s,t)|=0,
$$
with uniform convergence on sets of the form $\{-Rt^{1/2}\leq s\leq
Rt^{1/2}\}$, for any $R>0$.

\medskip

\noindent \textbf{Step 5. General $K>0$}. Recall that, by easyness
of writing, all the previous analysis has been done under the
assumption that $K=1$. In order to pass to general $K>0$, let $w_0$
be an initial condition satisfying \eqref{interm1} and $w(s,t)$ be
the corresponding solution to \eqref{PME}. Define then
$$
z(s,t)=\frac{1}{K}w(K^{(m-1)/2}s,t),
$$
which is another solution to \eqref{PME}, but this time with
$z(s,t)\to1$ as $s\to-\infty$, for any $t\geq0$. We apply the
previous step for $z(s,t)$, then undo the rescaling to get that
$$
\lim\limits_{t\to\infty}|w(s,t)-KW(K^{-(m-1)/2}s,t)|=0,
$$
uniformly on on sets of the form $\{-Rt^{1/2}\leq s\leq Rt^{1/2}\}$,
for any $R>0$.

\medskip

\noindent \textbf{Step 6. Back to initial variables. Behavior at the
origin.} Undoing the transformation in Section \ref{sec.prel} and
getting back to radial variables $(r,t)$, we have shown that
\begin{equation}\label{interm4}
\lim\limits_{t\to\infty}\left|u(r,t)-Kf\left(K^{-(m-1)/2}\frac{\log\,r}{\sqrt{t}}\right)\right|=0,
\end{equation}
uniformly in sets of the form $\{e^{-R\sqrt{t}}\leq r\leq
e^{R\sqrt{t}}\}$, for any $R>0$. It remains to prove that this
convergence holds true uniformly in the whole $\real^2$. To this
end, we follow the same strategy as in \cite[Section 5, Steps 5 and
6]{IS14}, by showing first the following

\medskip \noindent \textbf{Claim:} The value at point $x=0$ doesn't change, that is $u(0,t)=K$
for any $t>0$, if $u_0(0)=K>0$.

\begin{proof}[Proof of the Claim] In order to simplify the notation, let
\begin{equation}\label{interm5}
W_K(r,t)=Kf\left(K^{-(m-1)/2}\frac{\log\,r}{\sqrt{t}}\right),
\end{equation}
with $f$ being the profile introduced in \eqref{specialSS}. Arguing
by contradiction, assume $u_0(0)=K>0$ and there exists $t_0>0$ such
that $u(0,t_0)=K_1\in(0,K)$ (if $K_1>K$ the argument is identical).
Then, we can start the evolution at $t=t_0$ and the solution to
\eqref{eq1} with initial data $u(x,t_0)$ is $u(x,t+t_0)$ by
uniqueness. Applying \eqref{interm4} for both $u$ and
$u(\cdot+t_0)$, we find that
$$
\lim\limits_{t\to\infty}\left|u(r,t)-W_K(r,t)\right|=0
$$
and at the same time
$$
\lim\limits_{t\to\infty}\left|u(r,t+t_0)-W_{K_1}(r,t)\right|=0,
$$
with uniform convergence in sets of the form $\{e^{-R\sqrt{t}}\leq
r\leq e^{R\sqrt{t}}\}$, for any $R>0$. But the two assertions above
are contradictory, as the profiles $W_K$ and $W_{K_1}$ are
essentially different in sufficiently large sets.
\end{proof}
Showing the uniform convergence in \eqref{interm4} in the whole
space becomes now a standard fact, as it is immediate for radially
non-increasing solutions (using Lemma \ref{lem.prel}), and it
extends to general data (and solutions) by comparison at the level
of transformed $(s,t)$ variables. Details are totally identical to
those in Big Step A and Step 6 in \cite[Section 5]{IS14}, to which
we refer the reader.
\end{proof}

\section{Large time behavior in dimension $N=1$}\label{sec.N1}

We deal here with the most interesting and novel behavior of the
present work, the large time behavior for $m>1$ and in dimension
$N=1$. As stated in the Introduction, evolution in dimension $N=1$
departs strongly from the one in dimension $N\geq2$, but in order to
be understood, it is at the level of the transformed equation
\eqref{PMEconv} where one should look for qualitative differences.

\subsection{Large time behavior for initial data with
$u_0(0)=0$}\label{subsec.N10}

The aim of this subsection is to prove Theorem \ref{th.conv3}. This
is now rather simple, in view of the work already done in
\cite[Section 4]{IS14} to which we refer the reader, so we will be
quite brief.
\begin{proof}[Proof of Theorem \ref{th.conv3}]
Let $u$ be a solution to \eqref{eq1} with initial condition $u_0$ as
in Theorem \ref{th.conv3}. We recall the transformation
$w(s,t)=u(x,t)$, $s=\log\,|x|$, which leads to Eq. \eqref{PMEconv}
solved by $w$. By \eqref{idminus}, we obtain that $w_0(s)\to0$ as
$s\to-\infty$ and $w_0\in L^1(\real)$, since
$$
\int_{-\infty}^{\infty}w_0(s)\,ds=\int_{0}^{\infty}\frac{u_0(x)}{x}\,dx=M_0<\infty.
$$
Let us also recall the following explicit profile for Eq.
\eqref{PMEconv}, obtained at the beginning of \cite[Section 4]{IS14}
(see also \cite{LS97})
\begin{equation}\label{dobleuve}
W(s,t)=\left\{\begin{array}{ll}t^{-1/m}\left[\frac{1}{m}st^{-1/m}\right]^{1/(m-1)},
& {\rm for} \ s\in[0,kt^{1/m}),\\ 0, & {\rm
otherwise},\end{array}\right.
\end{equation}
where the constant $k=k(M_0)$ from the branching point is uniquely
chosen such that
$$
\int_{-\infty}^{\infty}W(s,t)\,ds=\int_{0}^{kt^{1/m}}W(s,t)\,ds=M_0.
$$
Since $w_0\in L^1$ and has mass $M_0$, we are in the same conditions
as in \cite[Theorem 1.4]{LS98} (for the special case $q=m$ in their
notation), whence we deduce that, for any $p\in[1,\infty)$ we have
$$
\lim\limits_{t\to\infty}t^{(p-1)/mp}\|w(\cdot,t)-W(\cdot,t)\|_{p}=0,
$$
where $W$ is the profile defined in \eqref{dobleuve}. Going back to
initial variables $(x,t)$, we readily obtain the convergence in
\eqref{pconv}. The convergence in the sense of graphs \eqref{conv21}
follows now identically as in \cite[Proof of Theorem 1.4, p.
234]{IS14}, thus we omit it here.
\end{proof}

\subsection{Large time behavior for initial data with
$u_0(0)=K>0$}\label{subsec.N1K}

We let again $w(s,t)=u(x,t)$, $s=\log\,|x|\in\real$, and we arrive
to the diffusion-convection equation \eqref{PMEconv} solved by $w$.
Moreover, $w_0(s)=w(s,0)=u(e^s,0)$, thus \eqref{idplus1} implies
\begin{equation}\label{interm6}
\lim\limits_{s\to-\infty}w_0(s)=K>0, \qquad
\lim\limits_{s\to\infty}w_0(s)=0, \qquad 0\leq w_0(s)\leq K, \ {\rm
for \ any} \ s\in\real,
\end{equation}
and in fact, if we begin from compactly supported initial data
$u_0$, then $w_0$ has an interface on the right. In our previous
paper \cite{IS14}, we were dealing with apparently similar $w_0$ but
where the limit $K>0$ was taking place as $s\to\infty$. We show
here, as a consequence of our analysis, that these two situations
are \emph{strikingly different} at the level of Eq. \eqref{PMEconv}:
while in the latter, an \emph{asymptotic simplification} was taking
place (see \cite[Section 5]{IS14}), in the former there is no such
phenomenon and solutions joining the effect of both diffusion and
convection appear, in the form of \emph{traveling waves}. This is
intuitively explained by the following:

\noindent $\bullet$ When $w_0(s)\to 0$ as $s\to-\infty$ and
$w_0(s)\to K$ as $s\to\infty$, there is no mass coming from
$-\infty$; thus, the effect of the diffusion becomes secondary and
the convection gains, \cite{IS14}.

\noindent $\bullet$ When (as in our case), we have the reversed
limits as in \eqref{interm6}, there is always mass entering via the
diffusion process, whence both diffusion and convection will play a
role. This leads to the idea of traveling waves (fronts advancing to
the right).

We thus put $w(s,t)=f(s-ct)$ in Eq. \eqref{PMEconv}, $c>0$ being the
"speed". Then $f$ solves the following ODE:
\begin{equation*}
-cf(\xi)=(f^m)'(\xi)-f^m(\xi), \qquad \xi:=s-ct,
\end{equation*}
which can be explicitly integrated to find
$$
f(\xi)=\left[c-K_0e^{(m-1)\xi/m}\right]_+^{1/(m-1)},
$$
hence the explicit family of solutions to \eqref{PMEconv} subject to
conditions \eqref{interm6} is given by
\begin{equation}\label{TWs}
W_{s_0}(s,t)=\left[c-e^{(m-1)(s-s_0-ct)/m}\right]_+^{1/(m-1)},
\qquad c=K^{m-1},
\end{equation}
where $s_0\in\real$ is a free parameter. Let us also notice that
$W_{s_0}$ are ordered, in the sense that $-\infty<s_{0,1}\leq
s_{0,2}<\infty$ implies $W_{s_{0,1}}(s,t)\leq W_{s_{0,2}}(s,t)$.

We thus have a one-parameter family of candidate profiles,
presenting the expected behavior. In order to prove asymptotic
convergence (in the uniform sense) to one profile of the form
\eqref{TWs}, we have to pass first through the $L^1$ convergence.
Let us then proceed with the rigorous proof.
\begin{proof}[Proof of Theorem \ref{th.conv4}]
Let $u_0$ be a radially symmetric initial condition as in the
statement of Theorem \ref{th.conv4} and $u(x,t)$ be the
corresponding solution to \eqref{eq1}, in dimension $N=1$. Then
$u(x,t)$ is radially symmetric for any $t>0$, according to Lemma
\ref{lem.prel}. Applying the transformation $w(s,t)=u(r,t)$,
$s=\log\,r$, we find a solution $w(s,t)$ to Eq. \eqref{PMEconv},
whose initial condition $w_0$ satisfies \eqref{interm6}. We then
infer from \cite[Theorem B]{OR} and \cite[Theorem A]{vDdG} that
there exists a unique $s_0\in\real$ such that
\begin{equation}\label{interm16}
\lim\limits_{t\to\infty}\|w(\cdot,t)-W_{s_0}(\cdot,t)\|_{1}=0,
\end{equation}
where $s_0$ is the unique translation parameter such that
\begin{equation}\label{interm17}
\int_{-\infty}^{\infty}(w_0(s)-W_{s_0}(s,0))\,ds=0.
\end{equation}
Moreover, one can readily deduce that
\begin{equation}\label{interm20}
\lim\limits_{s\to-\infty}w(s,t)=K, \quad {\rm for \ any} \ t>0.
\end{equation}
This follows by a direct comparison argument that we only sketch.
Since $0\leq w_0(s)\leq K$ for any $s\in\real$, by comparison one
gets $w(s,t)\leq K$ for any $t>0$ and $s\in\real$. Fix now $\e>0$
small. One can readily find profiles of the form \eqref{TWs} with
speed $c_{-}=(K-\e)^{m-1}$, and translated with a very small
parameter $s_0$, lying below $w_0$. It then follows that
$$
K-\e\leq
\liminf\limits_{s\to-\infty}w(s,t)\leq\limsup\limits_{s\to-\infty}w(s,t)\leq
K, \qquad {\rm for \ any} \ t>0,
$$
which gives \eqref{interm20}.

Undoing the transformation and coming back to initial variables, we
get from \eqref{interm16}
\begin{equation}\label{convL1}
\lim\limits_{t\to\infty}\|u(\cdot,t)-U_{x_0}(\cdot,t)\|_1=0,
\end{equation}
where $U_{x_0}$ is defined in \eqref{profile22}, $c=K^{m-1}$ and
$x_0=e^{s_0}>0$, with $s_0$ introduced in \eqref{interm17}. Let us
recall that $U_{x_0}$ has been obtained previously in \cite[Section
7]{GKKV}, only as an example of explicit solution. In order to pass
now from the convergence in $L^1$ to the convergence in
$L^{\infty}$, we use the following result of real analysis, taken
from \cite[Lemma 2.8]{vDdG} (see also \cite[Lemma 3]{Pel71}):
\begin{lemma}\label{lem.passing}
Let $\Phi:[0,\infty)\mapsto[0,\infty)$ be a function satisfying:
$\Phi(0)=0$, $\Phi\in L^1([0,\infty))$ and $\Phi$ is uniformly
Holder continuous on $[0,\infty)$ with exponent $\alpha$ and
constant $H>0$. Then
\begin{equation}\label{L1Linfty}
\|\Phi\|_{\infty}\leq
H^{1/(1+\alpha)}\left(\frac{1+\alpha}{\alpha}\|\Phi\|_1\right)^{\alpha/(1+\alpha)}.
\end{equation}
\end{lemma}
Moreover, we get from \eqref{interm20} that $u(0,t)=K$ for any
$t>0$. Since by \eqref{idplus} $u_0\in C^{0,\alpha}(\real)$,
$\alpha\in(0,1]$, we readily deduce from Lemma \ref{lem.Holder} and
Corollary \ref{cor.Lip} that the function
$\Phi(x):=|u(x,t)-U_{x_0}(x,t)|$ satisfies the conditions in Lemma
\ref{lem.passing} for any $t>0$. Joining \eqref{L1Linfty} and
\eqref{convL1}, we readily infer the uniform convergence
\eqref{conv22}, ending the proof.
\end{proof}
We end this subsection with a numerical experiment illustrating the
evolution towards translated profiles of the type $U_{x_0}(x,t)$, as
stated by Theorem \ref{th.conv4} and just proved above. In Figure
\ref{figure1} we show in parallel a generic solution to Eq.
\eqref{eq1} (taking for the experiment $u_0(x)=0.1(0.5-x^2)_+$),
respectively the corresponding explicit asymptotic profile.
\begin{figure}[ht!]
  \begin{center}
  \includegraphics[width=15cm,height=10cm]{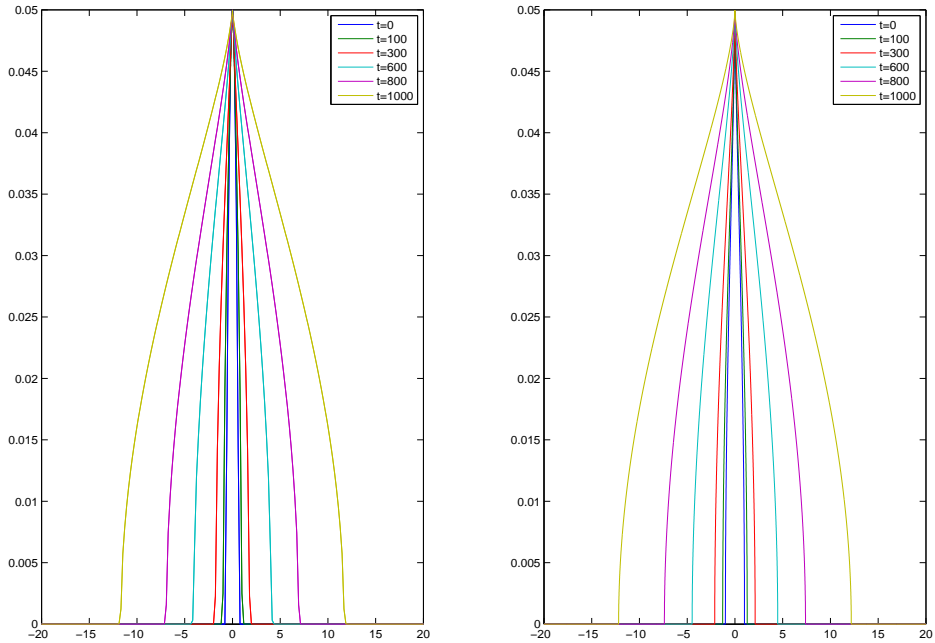}
  \end{center}
  \caption{A generic solution and the asymptotic profile for Eq. \eqref{eq1} in dimension $N=1$, with $m=3$.} \label{figure1}
\end{figure}

\section{Further extensions, comments and open
problems}\label{sec.final}

In this final section we extend some results and state some open
problems related to the present work. In particular, we discuss the
large time behavior for non-radial initial data in dimension $N=1$.

\medskip

\noindent \textbf{1. The linear case $m=1$.} This case is easy,
since it can be directly transformed into the standard Heat Equation
\eqref{HE} by the change of variable \eqref{transf.lin}. Thus, we
get from \cite{IS13} that the large time behavior is given by the
same asymptotic profiles as there, but in our case with
$N\in\{1,2\}$, namely (omitting the exact regularity assumptions on
$u_0$, which are the same as in \cite[Theorems 1.1 and 1.2]{IS13})

\medskip

\noindent $\bullet$ If $u_0(0)=0$, we have
\begin{equation}\label{asymptlin1}
\lim\limits_{t\to\infty}t^{1/2}\left\|u(\cdot,t)-\frac{M_{u_0}}{\omega}F(\cdot,t)\right\|_{\infty}=0,
\end{equation}
where $\omega=2\pi$ if $N=2$ and $\omega=1$ if $N=1$,
\begin{equation*}
M_{u_0}:=\int_{\real^N}|x|^{-N}u_0(x)\,dx<\infty,
\end{equation*}
and
\begin{equation}\label{profilelin1}
F(x,t):=\left\{\begin{array}{ll}\frac{1}{\sqrt{4\pi
t}}G\left(\frac{\log|x|+(N-2)t}{2\sqrt{t}}\right), \quad
G(\xi)=e^{-\xi^2}, & {\rm for} \ |x|\neq0, \\0, & {\rm for} \
x=0,\end{array}\right.
\end{equation}

\medskip

\noindent $\bullet$ If $u_0(0)=K>0$, we have
\begin{equation}\label{asymptlin2}
\left\|u(\cdot,t)-\frac{K}{2}E(\cdot,t)\right\|_{\infty}=O(t^{-1/2}),
\quad {\rm as} \ t\to\infty,
\end{equation}
where
\begin{equation}\label{profilelin2}
E(x,t):=\left\{\begin{array}{ll}{\rm
erfc}\left(\frac{\log|x|+(N-2)t}{2\sqrt{t}}\right), \quad {\rm
erfc}(\xi)=\frac{2}{\sqrt{\pi}}\int_{\xi}^{\infty}e^{-\theta^2}\,d\theta, & {\rm for} \ |x|\neq0,\\
K, & {\rm for} \ x=0.\end{array}\right.
\end{equation}

These statements are given here only at a formal level, for a
rigorous analysis see \cite{IS13}.

\medskip

\noindent \textbf{2. Non radially symmetric solutions.} This is a
natural question, as on the one hand, our techniques (specially in
dimension $N=2$) depend strongly on the symmetry of the solutions,
and on the other hand, in the linear case the authors were able to
show that the evolution do not symmetrize in the large time behavior
(see \cite[Introduction]{IS13} for a counterexample). In
\textbf{dimension $N=1$}, one can get a \textbf{full result for
general data $u_0$} by noticing that the origin is disconnecting the
real line. Thus, if $u_0$ is a general data, we can define the
radially symmetric data
$$
u_0^{+}(x)=\left\{\begin{array}{ll}u_0(x), & {\rm for} \ x\geq0, \\
u_0(-x), & {\rm for} \ x<0, \end{array}\right. \qquad
u_0^{-}(x)=\left\{\begin{array}{ll}u_0(x), & {\rm for} \ x\leq0, \\
u_0(-x), & {\rm for} \ x>0, \end{array}\right.
$$
and respectively, the corresponding radially symmetric solutions
$u^{+}(x,t)$, $u^{-}(x,t)$. To give an example, suppose that
$u_0(0)=K>0$ (the case $u_0(0)=0$ is similar), then
$u_0^+(0)=u_0^-(0)=K$, whence, by Theorem \ref{th.conv4}, there
exist $x_0^+$ and $x_0^-\in\real$ such that
\begin{equation}\label{interm18}
\lim\limits_{t\to\infty}\|u^+(\cdot,t)-U_{x_0^+}(\cdot,t)\|_{\infty}=\lim\limits_{t\to\infty}\|u^-(\cdot,t)-U_{x_0^-}(\cdot,t)\|_{\infty}=0.
\end{equation}
Moreover, since $u(0,t)=K$ for any $t>0$, solving the Cauchy problem
with $u_0$ reduces to solving two Dirichlet problems on the two
half-lines with boundary data $K>0$ and then matching the two
solutions. This shows that
\begin{equation}\label{interm19}
u(x,t)=\left\{\begin{array}{ll}u^+(x,t), & {\rm for} \
x\geq0,\\u^-(x,t), & {\rm for} \ x<0,\end{array}\right.
\end{equation}
for any $(x,t)\in\real\times(0,\infty)$. We thus infer from
\eqref{interm18} and \eqref{interm19} that
\begin{equation}\label{conv.nonradial}
\lim\limits_{t\to\infty}\|u(\cdot,t)-U(\cdot,t)\|_{\infty}=0, \qquad
U(x,t)=\left\{\begin{array}{ll}U_{x_0^+}(x,t), & {\rm for} \
x\geq0,\\U_{x_0^-}(x,t), & {\rm for} \ x<0,\end{array}\right.
\end{equation}
which is the desired large time behavior.

However, in dimension $N=2$ the previous argument does not work, and
we do not have for the moment an answer to the question whether
general data also lead to radially symmetric profiles or not. We
leave this as an \emph{open problem}. Our conjecture is that
non-radial asymptotic behavior should appear, in line with
\cite{IS13}.

\medskip

\noindent \textbf{3. Open problems for other $\gamma$ and $m$}. In
the papers \cite{KR04, R02}, a connection between equations such as
\eqref{eq1} in dimension $N=1$ (but with general density
$\varrho(x)=|x|^{-\gamma}$ for $\gamma=(3m+1)/2m$) and some
Fisher-KPP type equations is done. As a byproduct, the authors
deduce that the large time behavior in their cases is given by
self-similar, \emph{dipole-type} profiles. This is in itself a very
interesting result, but does not interact with ours, since
$\gamma=2$ implies $m=1$, while the transformation in \cite{KR04}
only applies to $m>1$. However, having these works as starting
point, we raise the open problem of studying the large time behavior
for the general form of \eqref{eq1}, that is
$$
|x|^{-\gamma}\partial_tu(x,t)=\Delta u^m(x,t),
$$
where the most interesting case is when $2<\gamma$ and $m>1$. It
seems to us that for the general case $(\gamma,m)$, there is no
useful mapping to another well-known equation; such mappings may
exist, but only for some special relations between $\gamma$ and $m$.

%%%%%%%%%%%%%%%%%%%%%%%%%%%%%%%%%%%%%%%%%%%%%%%%%%%%%%%%%%%%%%%%%%%%%%%%%%%%%%%%%%%

\section*{Acknowledgments} R. G. I. is partially supported by the
Spanish project MTM2012-31103 and the Severo Ochoa Excellence
project SEV-2011-0087. A. S. is partially supported by the Spanish
project MTM2014-53037-P, and he gratefully acknowledge financial
support from Universidad Rey Juan Carlos-Banco de Santander
Excellence group QUINANOAP. Both authors want to thank Prof.
Philippe Lauren\ced{c}ot for useful suggestions and for indicating
some important references.

\bibliographystyle{plain}

\end{document}